\documentclass{amsart}
\usepackage{hyperref,amsmath,amssymb,tikz}
\usetikzlibrary{arrows,shapes,positioning,decorations.markings, decorations.pathreplacing,decorations.pathmorphing}
\newcommand{\Cay}{\mathop{\mathrm{Cay}}}
\newcommand{\Aut}{\mathop{\mathrm{Aut}}}
\def\Zent#1{{\bf Z}({{#1}})}
\def\cent#1#2{{\bf C}_{{#1}}({{#2}})}

\newtheorem{theorem}{Theorem}[section]
\newtheorem{lemma}[theorem]{Lemma}
\newtheorem{corollary}[theorem]{Corollary}
\newtheorem{prop}[theorem]{Proposition}

\theoremstyle{definition}
\newtheorem{definition}[theorem]{Definition}
\newtheorem{notation}[theorem]{Notation}

\newcommand{\Orr}{\mathop{\mathrm{ORR}}}
\begin{document}
\title{Classification of finite groups that admit an oriented regular representation}  

\author{Joy Morris} 
\address{Joy Morris, Department of Mathematics and Computer Science \\
University of Lethbridge \\
Lethbridge, AB. T1K 3M4. Canada}
\email{joy.morris@uleth.ca}
\thanks{This research was supported in part by the National Science
  and Engineering Research Council of Canada grant RGPIN-2017-04905.}

\author{Pablo Spiga}
\address{Pablo Spiga, Dipartimento di Matematica e Applicazioni, University of Milano-Bicocca, Via Cozzi 55, 20125 Milano, Italy} 
\email{pablo.spiga@unimib.it}

\begin{abstract}
 This is the third, and last, of a series of  papers dealing with oriented regular representations. Here we complete the classification of finite groups that admit an oriented regular representation (or $\Orr$ for short), and  give a complete answer to a 1980 question of L\'{a}szl\'{o} Babai: ``Which [finite] groups admit an oriented graph as a DRR?" It is easy to see and well-understood that  generalised dihedral groups do not admit $\Orr$s. We prove that, with $11$ small exceptions (having orders ranging from $8$ to $64$), every finite group that is not  generalised dihedral  has an $\Orr$.
\end{abstract}

\smallskip

\begin{center}
\textit{To L\'{a}szl\'{o} Babai: for asking questions that keep us very entertained.}
\end{center}

\smallskip


\maketitle

\section{Introduction}\label{s: intro}
All groups and graphs in this paper are finite. 
Let $G$ be a  group and let $S$ be a subset of $G$. The \textbf{\emph{Cayley digraph}}, denoted by $\Cay(G,S)$, over $G$ with connection set $S$ is the digraph with vertex set $G$ and with $(x,y)$ being an arc if $yx^{-1}\in S$. (An \textbf{\em arc} is an ordered pair of adjacent vertices.) Since the group $G$ acts faithfully as a group of automorphisms of  $\Cay(G,S)$ via the right regular representation, Cayley digraphs represent  groups geometrically and combinatorially as groups of automorphisms of digraphs. Naively, the closer $G$ is to the full automorphism group of $\Cay(G,S)$, the closer this representation is to a precise graphical encoding of $G$.

Following this line of thought, it is natural to ask which  groups $G$ admit a subset $S$ with $G$ being the automorphism group of $\Cay(G,S)$; that is, $\Aut(\Cay(G,S))=G$. We say that $G$ admits a  \textbf{\em digraphical regular representation} (or DRR for short) if there exists a subset $S$ of $G$ with $\Aut(\Cay(G,S))=G$. Babai~\cite[Theorem~$2.1$]{babai1} has given a complete classification of the  groups admitting a DRR: except for
\begin{equation}\label{list}Q_8, \,\,C_2^2,\,\, C_2^3,\,\,   C_2^4\,\ \textrm{and}\,\,  C_3^2,
  \end{equation}every  group admits a DRR.

In light of Babai's result, it is natural to try to combinatorially represent  groups as automorphism groups of special classes of Cayley digraphs. Observe that, if $S$ is inverse-closed (that is, $S=S^{-1}:=\{s^{-1}\mid s\in S\}$), then $\Cay(G,S)$ is undirected. Now, we say that $G$ admits a \textbf{\em graphical regular representation} (or GRR for short) if there exists an inverse-closed subset $S$ of $G$ with $\Aut(\Cay(G,S))=G$. With a considerable amount of work culminating in \cite{Godsil,Hetzel}, the  groups admitting a GRR have been completely classified.  (In fact, this question attracted significant interest well before the DRR problem, although the final solutions to both problems appeared at about the same time.)

We recall that a \textbf{\em tournament} is a digraph $\Gamma:=(V,A)$ with vertex set $V$ and arc set $A$ such that, for every two distinct vertices $x,y\in V$, exactly one of  $(x,y)$ and $(y,x)$ is in $A$.  After the completion of the classification of DRRs and GRRs,  Babai and Imrich~\cite{babai2} proved that every group of odd order except for $C_3^2$ and $C_3^3$ admits a \textbf{\em tournament regular representation} (or TRR for short).  That is, each finite odd-order group $G$ different from $C_3^2$ and $C_3^3$ contains a subset $S$ with $\Cay(G,S)$ being a tournament and with $\Aut(\Cay(G,S))=G$. In terms of the connection set $S$, the Cayley digraph $\Cay(G,S)$ is a tournament if and only if $S\cap S^{-1}=\emptyset$ and $G\setminus\{1\}=S\cup S^{-1}$. This observation makes it clear that a Cayley digraph on $G$ cannot be a tournament if $G$ contains an element of order $2$, so only groups of odd order can admit TRRs.

In~\cite[Problem 2.7]{babai1}, Babai observed that one class of Cayley digraphs is rather interesting and  had not been investigated in the context of regular representations; that is, the class of oriented Cayley digraphs (or as Babai called them, oriented Cayley graphs). An \textbf{\em oriented Cayley digraph} is in some sense a ``proper" Cayley digraph. More formally, it is a Cayley digraph $\Cay(G,S)$ whose connection set $S$ has the property that $S \cap S^{-1}=\emptyset$. Equivalently, in graph-theoretic terms, it is a Cayley digraph with no digons. 
\begin{definition}
The group $G$ admits an \textbf{\em oriented regular representation} (or ORR for short) if there exists a subset $S$ of $G$ with $S \cap S^{-1}=\emptyset$ and $\Aut(\Cay(G,S))=G$. 
\end{definition}
Babai asked in~\cite{babai1} which (finite) groups admit an ORR.
Since a TRR is a special type of ORR, and $C_3^2$ is one of the five groups in Eq.~\ref{list} that do not admit a DRR (so cannot admit an ORR), the answer to this question for groups of odd order was already known when Babai published his question.

In this paper, answering the question of Babai and also confirming the conjecture given in~\cite[Conjecture~$1.5$]{morrisspiga}, we prove the following result.
\begin{theorem}\label{conj}
Every finite  group $G$ admits an $\Orr$, unless one of  the following holds:
\begin{description}
\item[(i)]$G$ is generalised dihedral with $|G|>2$ (see Definition~$\ref{def:gendih}$ for the meaning of generalised dihedral);
\item[(ii)]$G$ is isomorphic to one of the following eleven groups
\begin{align*}
&Q_8,\,C_4\times C_2,\, C_4\times C_2^2,\, C_4\times C_2^3,\, C_4\times C_2^4,\,C_3^2,\,C_3\times C_2^3,\\
&\langle a,b\mid a^4=b^4=(ab)^2=(ab^{-1})^2=1\rangle \textrm{ (of order $16$)},\\
&\langle a,b,c\mid a^4=b^4=c^4=(ba)^2=(ba^{-1})^2=(bc)^2=(bc^{-1})^2=1,\\
&\qquad\qquad a^2=c^2,a^c=a^{-1}, a^2=b^2\rangle \textrm{ (of order $16$)},\\
&\langle a,b,c\mid a^4=b^4=c^4=(ab)^2=(ab^{-1})^2=1,\\
&\qquad\qquad(ac)^2=(ac^{-1})^2=(bc)^2=(bc^{-1})^2=a^2b^2c^2=1\rangle \textrm{ (of order $32$)},\\
& D_4\circ D_4 \textrm{ (the central product of two dihedral groups of order $8$,} \\
&\qquad\qquad \textrm{which is the extraspecial group of order $32$ of plus type)}.
\end{align*}
\end{description}
\end{theorem}

We remark that since this theorem relies on results in the previous two papers \cite{morrisspiga,spiga} and since the results in \cite{morrisspiga} depend upon the Classification of Finite Simple Groups, this theorem also depends on the Classification.

In our opinion this is not the final word on oriented regular representations of finite groups. In fact, it is still unclear whether $\Orr$s behave asymptotically like DRRs and GRRs. (It is believed that most Cayley digraphs are DRRs and that most Cayley graphs are GRRs. One should be very careful about how to understand  ``most" in these statements and we refer the reader to the introduction of~\cite{DSV} for two distinct, natural interpretations of  ``most".)

We conclude this introductory section by observing that regular representations have shown a new vitality lately. For instance, Marston Conder, Mark Watkins and Tom Tucker~\cite{WT} have been studying finite groups admitting a graphical  Frobenius  representation and have posed some very intriguing conjectures in this context. All of these conjectures are in line with the classification of groups that admit DRRs, GRRs, TRRs and now ORRs: except for  some ``low level noise" (yielding a finite number of small exceptions) and any obvious obstructions, regular representations of the desired type will exist.  For DRRs, there are no obvious general obstructions; for TRRs, groups of even order are problematic and yield the only general obstruction; for GRRs, groups admitting automorphisms that map each element to itself or to its inverse are problematic and yield  the only general obstruction; for ORRs, groups for which every generating set contains at least one involution (that is, generalised dihedral groups) are problematic and (in light of Theorem~\ref{conj}) yield the only general obstruction.

Finally, we refer to~\cite{DSV,MSV,Spiga,XF} for some recent work on similar problems.

\section{Earlier work and preliminaries}\label{s: 2}
Before moving to the proof of Theorem~\ref{conj} we need to review the main results that have been proved on oriented regular representations. We start with a few definitions.

Babai pointed out in~\cite{babai1} that generalised dihedral groups of order greater than $2$ can never admit an ORR.  (Given a group element $g$, we denote by $o(g)$ its \textbf{{\em order}}.)
\begin{definition}\label{def:gendih}
Let $A$ be an abelian group. The \textbf{{\em generalised dihedral group}} over $A$ is the group $\langle \tau, A \rangle$ with $o(\tau)=2$ and $\tau a \tau=a^{-1}$ for every $a \in A$.
\end{definition}
In the special case where $A$ is cyclic, this is the dihedral group over $A$. Observe that, unless $|G|=2$, if $\Cay(G,S)$ is an ORR, then $\Cay(G,S)$ is connected and hence $S$ is a generating set for $G$. Now, Babai's observation follows immediately from the fact that if $G$ is the generalised dihedral group over the abelian group $A$, then every element of $G\setminus A$ has order $2$. Thus every generating set $S$ for $G$ must contain an involution, so that $S \cap S^{-1} \neq \emptyset$. This renders understanding generalised dihedral groups very important when we are studying ORRs.

Let $G$ be a finite group. As customary, we denote by $d(G)$ the \textbf{{\em minimum number of generators}} for $G$.
Following~\cite[Section~$2$]{spiga}, we say that a generating set $\{g_1,\ldots,g_d\}$ for $G$ is  \textbf{{\em irredundant}} if, for each $i\in \{1,\ldots,d\}$, the $d-1$ elements $$g_1,g_2,\ldots,g_{i-1},\,g_{i+1},\ldots,g_{d}$$ do not generate $G$. Observe that each generating set for $G$ of cardinality $d(G)$ is irredundant.

We say that the  $d$-tuple  $(g_1,\ldots,g_d)$  of elements of $G$ is \textbf{\textit{beautiful}} if the following conditions hold:
\begin{description}
\item[(i)]$\{g_1,\ldots,g_d\}$ is an irredundant generating set for $G$,
\item[(ii)]$o(g_i)>2$ for every $i\in \{1,\ldots,d\}$,
\item[(iii)]$o(g_{i+1}g_i^{-1})>2$ for every $i\in \{1,\ldots,d-1\}$.
\end{description}
Observe that being beautiful is a property of ordered tuples and not of sets; that is, it depends upon the ordering of the generating set $\{g_1,\ldots,g_d\}$ for $G$.

An important connection between beautiful generating tuples and $\Orr$s is given in the next theorem.

\begin{theorem}\label{thrmJoy}
Let $G$ be a finite group admitting a beautiful generating tuple. Then $G$ admits an $\mathrm{ORR}$ if and only if $G \not\cong Q_8$, $G \not\cong C_3 \times C_2^3$, and $G \not\cong C_3 \times C_3$. 
\end{theorem}
This theorem is implicit in~\cite{morrisspiga} and follows immediately from the theory developed therein. For a proof see~\cite[Theorem~$2.1$]{spiga}.

Developing the theory of beautiful generating tuples (and actually something more general, which we called five-product-avoiding generating sets), we have proved in \cite{morrisspiga} that each non-soluble group admits an $\Orr$. Building on this result, the second author has proved the following result, which (among other things) reduces the classification of groups admitting an $\Orr$ to some very specific infinite families of $2$-groups.

\begin{theorem}[(\cite{spiga}, Theorem 1.2)]\label{main}
  Let $G$ be a finite group. Then one of the following holds:
  \begin{description}
  \item[(i)]$G$ admits an $\mathrm{ORR}$;
    \item[(ii)]$G$ has an abelian $2$-subgroup $A$, a normal subgroup $N$ and two elements $g\in G\setminus N$ and $n\in N\setminus A$ with $A<N<G$, $|G:N|=|N:A|=2$, $g^2=1$, $n^g=n^{-1}$ and $a^g=a^{-1}$ for each $a\in A$;
    \item[(iii)]there exists a normal subgroup $N$ of $G$, $g\in G$ and $n_0\in N$ with $|G:N|=2$, $G=\langle N,g\rangle$, $g^2=1$, $N$ is a $2$-group and the action of $g$ by conjugation on $N$ inverts precisely half of the elements of $N$ and $N=H\cup n_0H$, where $H:=\{n\in N\mid n^g=n^{-1}\}$. Moreover, $N$ has no automorphism inverting more than half of its elements. (Every group $N$ that has an automorphism inverting half of its elements and no automorphism that inverts more is classified in~\cite{HeMa} by Hegarty and MacHale);  
\item[(iv)]$G$ is isomorphic to $Q_8$, to $C_3\times C_3$ or to $C_3\times C_2^3$;
\item[(v)]$G$ is generalised dihedral.
    \end{description}
  \end{theorem}
In view of this theorem, the classification of finite groups admitting an $\Orr$ is reduced to the groups in ${\bf (ii)}$ and ${\bf (iii)}$. We address these two families in Sections~\ref{s:ii} and~\ref{s:iii}, respectively.

We remark that the ``flavour" of many of the $\Orr$s that are produced in this paper is quite different from those produced in~\cite{morrisspiga} and~\cite{spiga}. Those papers focused on the use of beautiful generating tuples, so that for any two consecutive elements $g_i$ and $g_{i+1}$ of the generating tuple, the product $g_{i+1}g_i^{-1}$ has order greater than $2$. As we will see, a common situation for a group $G$ in the families we study in this paper is that $G$ contains an elementary abelian subgroup $B$ of high rank and low index. In order to generate such a group $G$, we require a large number of elements from at least one of the cosets $Bg$ of $B$. If we place any two elements $g_i:=b_1g$ and $g_{i+1}:=b_2g$ of $Bg$ (where $b_1, b_2 \in B$) consecutively in the generating set, then $g_{i+1}g_i^{-1}=b_2gg^{-1}b_1^{-1}=b_2b_1^{-1} \in B$, so $o(g_{i+1}g_i^{-1}) \le 2$. In principle it could be possible to ensure that whenever $g_i \in Bg$ we have $g_{i+1} \notin Bg$,  but in practice we may still find that $o(g_{i+1}g_i^{-1})=2$. We will therefore take a very different approach that uses a GRR for $B$ as a starting point. In our previous papers, if we look at the induced subdigraph of each of our $\Orr$s on the vertices that lie in the connection set (equivalently, the induced subdigraph on the open neighbourhood of any vertex), that digraph is (weakly) connected and asymmetric. The $\Orr$s we produce in this paper will often include many isolated vertices in that induced subdigraph.

Before moving into our analysis of the groups in ${\bf (ii)}$ and ${\bf (iii)}$ (from above), we introduce some other results from the literature that will be important in our proofs.

The following result is found in the proof of the theorem in~\cite{Imrich2}. Note that the statement in~\cite{Imrich2} assumes only $k \ge 5$, but there is a mistake in the case $k=5$ that has been pointed out by multiple researchers. Although the elementary abelian $2$-group of rank $5$ also admits a GRR, it requires a different connection set so we omit it from our statement.

\begin{lemma}[(Imrich \cite{Imrich2})]\label{Imrich-2-gps}
Let $G$ be an elementary abelian $2$-group of rank $k \ge 6$, and let $\{x_1, \ldots, x_k\}$ be a generating set for $G$. Then $G$ has a $\mathrm{GRR}$; furthermore, a $\mathrm{GRR}$ is given by the Cayley graph on $G$ whose connection set consists of the $2k+1$ elements: 
\begin{equation}\label{im}x_1,\, \ldots,\, x_k,\, x_1x_2,\, x_2x_3,\, \ldots,\, x_{k-1}x_k,\, x_1x_2x_{k-2}x_{k-1},\,x_1x_2x_{k-1}x_k.
  \end{equation}
\end{lemma}

Note that the connection set found in this lemma depends on the order as well as the choice of the elements in the generating set for $G$. We therefore use the following definition.

\begin{definition}
Let $G$ be an elementary abelian $2$-group of rank $k \ge 6$. Given the generating tuple $(x_1, \ldots, x_k)$ for $G$, we refer to the connection set given in Eq.~\ref{im} as the \textbf{\emph{\mathversion{bold}Imrich generating set for $G$ with respect to $(x_1, \ldots, x_k)$}}. 
\end{definition}

In their work on the GRR problem, Nowitz and Watkins proved a lemma that is very useful in our context also. (Given a graph $\Gamma$ and a vertex $v$ of $\Gamma$, we denote by $\Aut(\Gamma)$ the automorphism group of $\Gamma$ and by $\Aut(\Gamma)_v$ the stabiliser of the vertex $v$ in $\Aut(\Gamma)$.)

\begin{lemma}[(Nowitz and Watkins~\cite{NW})]\label{Watkins-Nowitz}
Let $G$ be a group, let $S$ be a subset of $G$, let $\Gamma:=\Cay(G,S)$ and let $X$ be a subset of $G$. If $\varphi$ fixes $X$ pointwise for every $\varphi\in \Aut(\Gamma)_1$, then $\varphi$ fixes $\langle X\rangle$ pointwise for every $\varphi\in \Aut(\Gamma)_1$. 
\end{lemma}

Thus, if $\Aut(\Gamma)_1$ fixes every element of a generating set for $G$ or if the subgraph induced by $\Gamma$ on the neighbourhood $\Gamma(1)=S$ is asymmetric, then $\Aut(\Gamma)= G$. Hence $\Gamma$ is a DRR for $G$, and is therefore an ORR if the connection set satisfies $S\cap S^{-1}=\emptyset$. We will use this fact repeatedly when we cite the above lemma. Also, although Nowitz and Watkins did not make this explicit, the same proof applies if we replace both occurrences of the word ``pointwise" in the statement of Lemma~\ref{Watkins-Nowitz} with the word ``setwise." We will also use this sometimes when we cite the above lemma.

We include in this section two more lemmas that we will need. Lemma~\ref{abelian-orr} follows fairly easily from the work in~\cite{morrisspiga}, but we include a complete proof as the precise connection set (and therefore the fact that the induced subgraph on that connection set is weakly connected) is not easy to see from the statements in that paper.

\begin{lemma}\label{abelian-orr}
Let $A$ be an abelian $2$-group of order $2^k$. Assume that $A$ is not elementary abelian and $A \ncong C_4\times C_2^{k-2}$. Then there is a generating set $S$ for $A$ with $|S|\ge 2$ such that the induced subgraph of $\Gamma:=\Cay(G,S)$ on $S$ is weakly connected with trivial automorphism group, so that $\Gamma$ is an $\Orr$. 
\end{lemma}

\begin{proof}
From the structure of finite abelian groups, we may write $A=A_1\times \cdots \times A_m$, where $A_i:= \langle a_i \rangle$ is a non-trivial cyclic group for each $i\in \{1,\ldots, m\}$. Moreover, we may assume that $o(a_{i+1})$ divides $o(a_i)$ for every $i\in\{1,\ldots,  m-1\}$.

Since $A$ is not elementary abelian and $A \ncong C_4 \times C_2^{k-2}$, we must have either $o(a_1)>4$, or $m \ge 2$ and $o(a_1)=o(a_2)=4$. In the first case, set $x_1:=a_1$, $x_2:=a_1^{-1}a_2$ and  $x_i:=a_1^{(-1)^{i-1}}a_i$ for every $i\in \{3,\ldots,m\}$; in the second case, set $x_1:=a_1$, $x_2:=a_2$, $x_{2i+1}:=a_1a_{2i+1}$ and $x_{2j}:=a_2a_{2j}$ for every $2i+1,2j\in \{3,\ldots,m\}$. In both cases, it is easy to verify that the set $X:=\{x_1, \ldots, x_m\}$ generates $A$ irredundantly. (In fact, $(x_1, \ldots, x_m)$ is a beautiful generating tuple for $A$.)

Consider $S:=X \cup Y$ where 
\[
Y:=
\begin{cases}\{x_1^2\} &\textrm{if }m=1,\\
\{x_2x_1^{-2}\}\cup\{x_2x_1^{-1},x_3x_2^{-1},\ldots,x_m x_{m-1}^{-1}\}&\textrm{if }m\ge 2 \textrm{ and }o(a_1)>4,\\
  \{x_1x_2\}\cup\{x_2x_1^{-1},x_3x_2^{-1},\ldots,x_m x_{m-1}^{-1}\}&\textrm{if }m\ge 2 \textrm{ and } o(a_1)=o(a_2)=4.
\end{cases}
\]
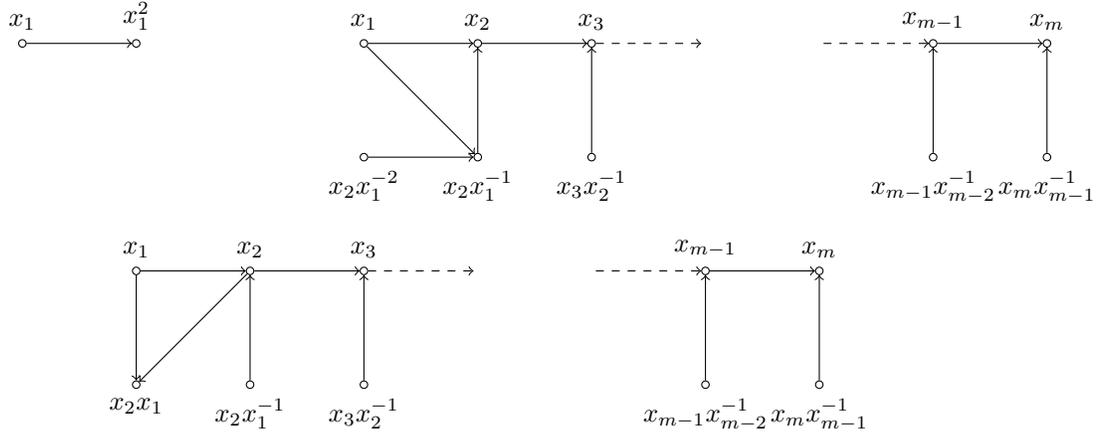
\begin{figure}[!hhh]
\begin{center}
\begin{tikzpicture}[node distance=1.4cm]
\node[circle,draw,inner sep=1pt, label=90:$x_1^2$](A0){};
\node[left=of A0,circle,draw,inner sep=1pt, label=90:$x_1$](A1){};
\draw[->](A1) to  (A0);
\node[right=of A0,circle,inner sep=1pt, label=90:](A33){};
\node[right=of A33,circle,draw,inner sep=1pt, label=90:$x_1$](A3){};
\node[right=of A3,circle,draw,inner sep=1pt, label=90:$x_2$](A4){};
\node[right=of A4,circle,draw,inner sep=1pt, label=90:$x_3$](A5){};
\node[right=of A5,circle,inner sep=1pt, label=90:](A6){};
\node[below=of A6,circle,inner sep=1pt, label=90:](B6){};
\node[right=of A6,circle,inner sep=1pt, label=-90:](AA6){};
\node[right=of AA6,circle,draw,inner sep=1pt, label=90:$x_{m-1}$](A7){};
\node[right=of A7,circle,draw,inner sep=1pt, label=90:$x_m$](A8){};
\node[below=of A3,circle,draw,inner sep=1pt, label=-90:$x_2x_1^{-2}$](B1){};
\node[below=of A4,circle,draw,inner sep=1pt, label=-90:$x_2x_1^{-1}$](B2){};
\node[below=of A8,circle,draw,inner sep=1pt, label=-90:$x_mx_{m-1}^{-1}$](B8){};
\node[below=of A5,circle,draw,inner sep=1pt, label=-90:$x_3x_{2}^{-1}$](B3){};
\node[below=of A7,circle,draw,inner sep=1pt, label=-90:$x_{m-1}x_{m-2}^{-1}$](B7){};
\node[below=of AA6,circle,inner sep=1pt, label=-90:](BB6){};
\draw[->](A3) to  (A4);
\draw[->](A4) to  (A5);
\draw[dashed,->](A5) to  (A6);
\draw[dashed,->](AA6) to (A7);
\draw[->](B7) to (A7);
\draw[->](A7) to  (A8);
\draw[->](B1) to  (B2);
\draw[->](A3) to  (B2);
\draw[->](B3) to  (A5);
\draw[->](B2) to (A4);
\draw[->](B8) to (A8);
\node[below=of A0,circle,inner sep=1pt, label=90:](xx){};
\node[below=of xx,circle,draw,inner sep=1pt, label=90:$x_1$](xA3){};
\node[right=of xA3,circle,draw,inner sep=1pt, label=90:$x_2$](xA4){};
\node[right=of xA4,circle,draw,inner sep=1pt, label=90:$x_3$](xA5){};
\node[right=of xA5,circle,inner sep=1pt, label=90:](xA6){};
\node[below=of xA6,circle,inner sep=1pt, label=90:](xB6){};
\node[right=of xA6,circle,inner sep=1pt, label=-90:](xAA6){};
\node[right=of xAA6,circle,draw,inner sep=1pt, label=90:$x_{m-1}$](xA7){};
\node[right=of xA7,circle,draw,inner sep=1pt, label=90:$x_m$](xA8){};
\node[below=of xA3,circle,draw,inner sep=1pt, label=-90:$x_2x_1$](xB1){};
\node[below=of xA4,circle,draw,inner sep=1pt, label=-90:$x_2x_1^{-1}$](xB2){};
\node[below=of xA8,circle,draw,inner sep=1pt, label=-90:$x_mx_{m-1}^{-1}$](xB8){};
\node[below=of xA5,circle,draw,inner sep=1pt, label=-90:$x_3x_{2}^{-1}$](xB3){};
\node[below=of xA7,circle,draw,inner sep=1pt, label=-90:$x_{m-1}x_{m-2}^{-1}$](xB7){};
\node[below=of AA6,circle,inner sep=1pt, label=-90:](xBB6){};
\draw[->](xA3) to  (xA4);
\draw[->](xA4) to  (xA5);
\draw[dashed,->](xA5) to  (xA6);
\draw[dashed,->](xAA6) to (xA7);
\draw[->](xB7) to (xA7);
\draw[->](xA7) to  (xA8);
\draw[->](xA4) to  (xB1);
\draw[->](xA3) to  (xB1);
\draw[->](xB3) to  (xA5);
\draw[->](xB2) to (xA4);
\draw[->](xB8) to (xA8);
\end{tikzpicture}
\end{center}
\caption{The oriented graph $\Delta$ in the proof of Lemma~\ref{abelian-orr}: top left when $m=1$; top right when $m\ge 2$ and $o(a_1)>4$; and bottom when $m\ge 2$ and $o(a_1)=o(a_2)=4$}\label{Fig1}
\end{figure}
Let $\Gamma:=\Cay(A,S)$ and observe that $S \cap S^{-1}=\emptyset$ so that $\Gamma$ is an oriented digraph. Let $\Delta$ be the subgraph induced by $\Gamma$ on $S$. It is clear that $(x_{i-1},x_i)$ and $(x_ix_{i-1}^{-1},x_i)$ are arcs of $\Delta$ for every $ i\in\{2,\ldots,  m\}$. Moreover, when $m=1$, $(x_1,x_1^2)$ is an arc of $\Delta$; when $m\ge 2$ and $o(a_1)>4$, $(x_1,x_2x_1^{-1})$ and $(x_2x_1^{-2},x_2x_1^{-1})$ are arcs of $\Delta$; while, when $m\ge 2$ and $o(a_1)=o(a_2)=4$, $(x_1,x_1x_2)$ and $(x_2,x_1x_2)$ are arcs of $\Delta$. (See Figure~\ref{Fig1}.) This shows that $\Delta$ is weakly connected. It is not hard (but rather tedious because it requires some detailed computations) to show that these are the only arcs of $\Delta$. We do not give a complete proof of this easy fact, but  we deal with one case to show the type of computations that are required. Suppose that there exists an arc $(x_ix_{i-1}^{-1},x_jx_{j-1}^{-1})$ between two distinct vertices in $Y$, for some $i,j\in \{2,\ldots,m\}$. Then $x_jx_{j-1}^{-1}x_{i-1}x_i^{-1}\in S$ and hence either  $x_jx_{j-1}^{-1}x_{i-1}x_i^{-1}=x_k$ for some $k\in \{1,\ldots,m\}$, or  $x_jx_{j-1}^{-1}x_{i-1}x_i^{-1}=x_kx_{k-1}^{-1}$ for some $k\in \{2,\ldots,m\}$, or $x_jx_{j-1}^{-1}x_{i-1}x_i^{-1}\in 	\{x_1^2,x_1x_2,x_2x_1^{-2}\}$. Each of these relations contradicts the irredundancy of the generating set $X$.

The structure of the arcs that we have described in $\Delta$ (see again Figure~\ref{Fig1}) shows that $\Delta$ has trivial automorphism group. Lemma~\ref{Watkins-Nowitz} then implies that $\Gamma$ is an $\Orr$.
\end{proof}

We conclude this section with a technical lemma. Although we will only apply this lemma with $|G:N| \le 8$, we state it in greater generality as the proof is no different and the general result may be useful in future work on $\Orr$s.
\begin{lemma}\label{l:1}
  Let $G$ be a $2$-group and suppose that $G$ is not generalised dihedral. Let $N$ be a normal subgroup of $G$   with $G/N$ elementary abelian. Suppose that there exists $T\subseteq N$ with $|T|\ge 2$ such that  $\Cay(N,T)$ is an $\Orr$ and such that the subgraph induced by $\Cay(N,T)$ on the neighbourhood $T$ of the vertex $1$ is weakly connected. Let $\kappa:= d(G/N)$. Then there is a set of elements $\{a_1, \ldots, a_\kappa\} \subseteq G \setminus N$ such that:
 \begin{itemize}
 \item $G=\langle a_1, \ldots, a_\kappa, N\rangle$; and
 \item $o(a_i)>2$ for each $i \in \{1, \ldots, \kappa\}$.
 \end{itemize} 
Furthermore, if there exists such a set of elements with the additional property that, for every $i,j \in \{1, \ldots, \kappa\}$ with $i\neq j$, we have $a_i^2$ centralises $N$, and either $a_i^2 \neq a_j^2$ or $a_ia_j$ does not centralise $N$, then $\Cay(G,T\cup \{a_1,\ldots,a_\kappa\})$ is an $\mathrm{ORR}$
\end{lemma}
\begin{proof}
Let $T\subseteq N$  satisfying the hypothesis of this lemma; that is: $|T| \ge 2$, $\Cay(N,T)$ is an $\Orr$, and the subgraph induced by $\Cay(N,T)$ on the neighbourhood $T$ of the vertex $1$ is weakly connected.

We begin by proving the first conclusion of this lemma. (The argument here is similar to the proof of \cite[Lemma $2.6$]{morrisspiga} and uses the ideas therein.) Among all $\kappa$-tuples $a_1,\ldots,a_\kappa$ of elements of $G$ such that $\{a_1N,\ldots,a_{\kappa}N\}$ is a generating set for $G/N$, choose one with as few involutions as possible. If no element in $\{a_1,\ldots,a_\kappa\}$ is an involution, then our first conclusion is proved. Hence we suppose that $\{a_1,\ldots,a_\kappa\}$ has at least one involution. Relabeling the index set $\{1,\ldots,\kappa\}$ if necessary, we may assume that $a_1$ is an involution. 

Let $n\in N$. Now, $\{a_1n,a_2,\ldots,a_\kappa\}$ is still a generating set for $G$ modulo $N$ of cardinality $\kappa$. Since this generating set cannot contain fewer involutions than the original generating set, the element $a_1n$ must be an involution. Thus $1=(a_1n)^2=a_1na_1n=n^{a_1}n$, that is, $n^{a_1}=n^{-1}$. Since this argument does not depend upon $n\in N$, we obtain that $a_1$ acts by conjugation inverting each element of $N$. In  particular, $N$ is abelian.

Let $j\in \{2,\ldots,\ell\}$ and let $n\in N$. Now, $\{a_1a_jn,a_2,a_3,\ldots,a_{\kappa}\}$ is still a generating set for $G$ modulo $N$ of cardinality $\kappa$. Since this generating set cannot contain fewer involutions than the original generating set,  the element $a_1a_jn$ must be an involution. Thus
$$1=(a_1a_jn)^2=a_1a_jna_1a_jn=a_1^2(a_jn)^{a_1}a_jn=a_j^{a_1}n^{a_1}a_jn=a_j^{a_1}n^{-1}a_jn;$$
so $a_j^{a_1}=n^{-1}a_{j}^{-1}n$. By applying this equality with $n:=1$, we deduce that conjugation by $a_1$ inverts $a_j$. Therefore $a_j^{-1}=a_j^{a_1}=n^{-1}a_j^{-1}n$ for each $n\in N$, that is, $a_j$ commutes with $N$.

Let $i,j\in \{2,\ldots,\kappa\}$ with $i\neq j$. Arguing as above, $\{a_1a_ia_j,a_2,a_3,\ldots,a_{\ell}\}$ is still a generating set for $G$ modulo $N$ of cardinality $\kappa$. Since this generating set cannot contain fewer involutions than the original generating set,  the element $a_1a_ia_j$ must be an involution. Thus
$$1=(a_1a_ia_j)^2=a_1a_ia_ja_1a_ia_j=a_1^2(a_ia_j)^{a_1}a_ia_j=(a_i^{a_1}a_j^{a_1})a_ia_j=(a_i^{-1}a_j^{-1})a_ia_j,$$
so $a_ia_j=a_{j}a_i$; that is, $a_i$ and $a_j$ commute. 

This shows that $M:=\langle N,a_2,\ldots,a_\ell\rangle$ is an abelian normal subgroup of $G$. Since $G=\langle M,a_1\rangle$ and $a_1$ has order $2$, we have $|G:M|=2$. Moreover, since the action of $a_1$ by conjugation inverts a generating set for $M$, we see that $G$ is a generalised dihedral group over $M$, contrary to our assumption. Our first conclusion is thus proven.

\smallskip

To complete the proof, we assume from now on that $\{a_1, \ldots, a_\kappa\}$ satisfies the additional hypothesis that, for every $i,j \in \{1, \ldots, \kappa\}$ with $i \neq j$,  we have $a_i^2$ centralises $N$, and either: $a_i^2 \neq a_j^2$, or $a_ia_j$ does not centralise $N$.

Consider the set $S:=T\cup \{a_1,\ldots,a_\kappa\}$. By construction, $S\cap S^{-1}=\emptyset$ and hence $\Gamma:=\Cay(G,S)$ is an oriented digraph. When $\kappa=0$, that is, $G=N$, $\Gamma$ is an $\Orr$ as part of our hypothesis. Thus, for the rest of the proof, we assume that $\kappa\ge 1$.

Let $\varphi$ be an automorphism of $\Gamma$ with $1^\varphi=1$. Now, $\varphi$ fixes the neighbourhood $\Gamma(1)=S$ of the vertex $1$, that is, $\varphi$ acts as a group of automorphisms of the subgraph $\Delta$ induced by $\Gamma$ on $S$. By construction, the vertices $a_1,\ldots,a_\kappa$ are isolated in $\Delta$. As $|T|\ge 2$ and $a_1,\ldots,a_\kappa$ are isolated in $\Delta$,  $T$ is the (unique) connected component of $\Delta$ of largest possible order. Thus $\varphi$ fixes $T$ setwise. Therefore, by Lemma~\ref{Watkins-Nowitz}, $\varphi$ fixes $\langle T\rangle=N$ setwise. Thus $\varphi$ acts as an automorphism of the graph induced by $\Gamma$ on $N$; that is, on the Cayley digraph $\Cay(N,S\cap N)=\Cay(N,T)$. By hypothesis, $\Cay(N,T)$ is an $\Orr$ and hence $\varphi$ fixes $N$ pointwise. Furthermore, this shows that, for every $n \in N$ and every vertex $v$ of $\Gamma$, we have 
\begin{equation}\label{n-fixes}
(nv)^\varphi=nv^\varphi. 
\end{equation}

Suppose that, for some $i \in \{1, \ldots, \kappa\}$, we have $a_i^\varphi\neq a_i$. Since $a_i^\varphi \in S$, there exists $j \in \{1, \ldots, \kappa\}\setminus \{i\}$ with $a_i^\varphi=a_j$. 
Now $(a_i,a_i^2)$ is an arc of $\Gamma$ and hence $(a_i^\varphi, (a_i^2)^\varphi)$ is also an arc of $\Gamma$. Since $G/N$ is elementary abelian, we have $a_i^2\in N$ and, since $\varphi$ fixes $N$ pointwise, we have $(a_i^2)^{\varphi}=a_i^2$. Therefore, $a_i^2(a_i^\varphi)^{-1}=a_i^2 a_j^{-1}\in S\setminus N=\{a_1,\ldots,a_\kappa\}$. Since $d(G/N)=\kappa$, this forces $a_i^2 a_j^{-1}=a_j$, so $a_i^2=a_j^2$.

By Eq.~\eqref{n-fixes}, $(na_i)^\varphi=na_j$ for every $n\in N$. Fix $n\in N$, and observe that $a_ina_i\in N$ because $G/N$ is elementary abelian and that $(a_ina_i)^\varphi=a_ina_i$ because $\varphi$ fixes $N$ pointwise. Since $(na_i,a_ina_i)$ is an arc of $\Gamma$, $((na_i)^\varphi,(a_ina_i)^\varphi)=(na_j,a_ina_i)$ is also an arc of $\Gamma$; that is, $$a_ina_i(na_j)^{-1}=a_i^{2}(a_i^{-1}na_i)(a_j^{-1}n^{-1}a_j)a_j^{-1}=a_i^2n^{a_i}(n^{-1})^{a_j}a_j^{-1}\in S.$$ As $a_i^2n^{a_i}(n^{-1})^{a_j}a_j^{-1}$ lies in the coset $Na_j^{-1}$ and $G/N$ is elementary abelian, we deduce $a_i^2n^{a_i}(n^{-1})^{a_j}a_j^{-1}=a_j$; that is, $n^{a_i}=n^{a_j}$. Thus $n^{a_ia_j}=(n^{a_i})^{a_j}=n^{a_j^2}=n$ because $a_j^2$ centralises $N$ by hypothesis. As this argument does not depend upon $n\in N$,  $a_ia_j$ centralises $N$, contradicting our hypothesis.

Since $i$ was arbitrary and our only assumption was that $(a_i)^\varphi \neq a_i$, we conclude that $\varphi$ fixes $S$ pointwise, so by Lemma~\ref{Watkins-Nowitz} we have $\varphi=1$ and $\Gamma$ is an $\Orr$.
\end{proof}

\section{Constructing $\Orr$s in groups that have low-index elementary abelian subgroups}

In this section, we prove some results that will be useful for constructing $\Orr$s on many groups that have elementary abelian subgroups of low index. These results make intense use of Imrich generating sets for GRRs on elementary abelian groups of rank at least $6$, so they require the elementary abelian subgroup to be of sufficiently high rank.

Given a Cayley digraph $\Cay(G,S)$ and two vertices $x$ and $y$, we say that $x$ and $y$ are \textbf{\mathversion{bold}adjacent via $s$} if $yx^{-1}=s \in S$. Similarly, we say that $x$ and $y$ are \textbf{\mathversion{bold}adjacent via $X$} if $yx^{-1} \in X\subset S$.

\begin{lemma}\label{mut-innbrs}
Let $B$ be an elementary abelian $2$-group of rank $k \ge 6$, and let $T$ be the Imrich generating set for $B$ (with respect to some generating tuple). Let $G$ be a group with $B \le G$, let  $Bx$ be a coset of $B$ in $G$ and let $S$ be any subset of $G$ with $S \cap Bx=Tx\cup \{x\}$. Then in $\Cay(G,S)$, $x$ is the only vertex of $S \cap Bx$ that has at least three mutual inneighbours via elements of $S \cap Bx$ with every other vertex of $S \cap Bx$.
\end{lemma}

\begin{proof}
Let the generating tuple be $(z_1, \ldots, z_k)$.
We begin by showing that $x$ has at least three common inneighbours via elements of  $S\cap Bx$ with every other vertex of $S \cap Bx$. Let $tx$ be an arbitrary vertex of $(S \cap Bx)\setminus\{x\}$, so $t \in T$. Observe that $y$ is a mutual inneighbour via elements of $S \cap Bx$ of $x$ and $tx$ if and only if there exist $t_1,t_2 \in T\cup \{1\}$ such that $t_1xy=x$ and $t_2xy=tx$. Equivalently, $y=t_1^x$ and $t_1t_2=t$. This is clearly satisfied with $y=t_1=1$ and $t_2=t$, so that $1$ is a mutual inneighbour. It is also always satisfied by taking $y=t^{x}$, $t_1=t$, and $t_2=1$, so that $t^{x}$ is a mutual inneighbour. The third mutual inneighbour depends on $t$.

If $t=z_i$ for some $ i\in\{1,\ldots,k-1\}$, then taking $t_1:=z_iz_{i+1}$, $t_2:=z_{i+1}$, and $y:=t_1^{x}$ gives $y$ as a third mutual inneighbour. If $t=z_k$, then taking $t_1:=z_{k-1}z_k$, $t_2:=z_{k-1}$, and $y:=t_1^{x}$ gives $y$ as a third mutual inneighbour.
If $t=z_iz_{i+1}$ for some $i\in\{1,\ldots k-1\}$, then taking $t_1:=z_i$, $t_2:=z_{i+1}$, and $y:=t_1^{x}$ gives $y$ as a third mutual inneighbour.
Finally, if $t=z_1z_2z_{\ell-1}z_{\ell}$ for $\ell \in \{k-1,k\}$, then taking $t_1:=z_1z_2$, $t_2:=z_{\ell-1}z_{\ell}$, and $y:=t_1^{x}$ gives $y$ as a third mutual inneighbour.

Now we show that for any vertex $tx$ of $Tx$, there is some vertex $t'x$ of $Tx\setminus\{tx\}$ such that $tx$ and $t'x$ have fewer than three common inneighbours via elements of $S \cap Bx$. Observe that $y$ is a mutual inneighbour via elements of $S \cap Bx$ of $tx$ and $t'x$ if and only if there exist $t_1,t_2 \in T\cup \{1\}$ such that $t_1xy=tx$ and $t_2xy=t'x$. Equivalently, $y=(t_1t)^x$ and $t_1t_2=tt'$. This is clearly satisfied with $y=1$, $t_1=t$, and $t_2=t'$. It is also satisfied with $y=(tt')^{x}$, $t_1=t'$, and $t_2=t$. Thus $1$ and $(tt')^{x}$ are mutual inneighbours of $tx$ and $t'x$. 

If $t=z_i$ for some $ i\in\{1,\ldots, k\}$, then let $t':=z_j$ for some $j \notin\{i-2,i-1,i,i+1,i+2\}$ (such a $j$ exists since $k \ge 6$). There is no way to write $z_iz_j$ as a product of two elements $t_1, t_2 \in T\cup\{1\}$ except when $\{t_1,t_2\}=\{z_i,z_j\}$. Thus there is no third mutual inneighbour of $z_ix$ and $z_jx$ via elements of $S \cap Bx$.
Similarly, if $t=z_iz_{i+1}$, then let $t'$ be either $z_{i+2}z_{i+3}$ when $i \le k-3$, or $z_{i-2}z_{i-1}$ otherwise. Again, if $tx$ and $t'x$ had a third mutual inneighbour, then there would be some other way of writing some $z_{j}z_{j+1}z_{j+2}z_{j+3}$  (where $j \in \{i,i-2\}$) as a product of two elements of $T \cup \{1\}$, but this is not possible. 
Finally, if $t=z_1z_2z_{\ell-1}z_{\ell}$ where $\ell \in \{k-1,k\}$, then let $t':=z_3$. There is no other way of writing $z_1z_2z_3z_{\ell-1}z_{\ell}$ as a product of two elements of $T \cup \{1\}$. This completes the proof. 
\end{proof}

Unfortunately, to make the above lemma directly useful in building $\Orr$s, we would need to know that every automorphism of $\Cay(G,S)$ that fixes $1$ also fixes $S \cap Bx=Tx \cup \{x\}$ setwise. This is challenging to prove in general. It turns out to be much easier to prove that if $S \cap Bx= (Bx \setminus Tx)\setminus\{x\}$, then every automorphism of $\Cay(G,S)$ that fixes $1$ also fixes $S \cap Bx$. We therefore prove the following result about the case where $S \cap Bx= (Bx \setminus Tx)\setminus\{x\}$, which is a corollary to the above lemma.

\begin{corollary}\label{cor-innbrs}
Let $B$ be an elementary abelian $2$-group of rank $k \ge 6$, and let $T$ be the Imrich generating set for $B$ (with respect to some generating tuple). Let $G$ be a group with $B \le G$, let $Bx$ be a coset of $B$ in $G$ and let $S$ be any subset of $G$ with $S \cap Bx=(Bx\setminus Tx)\setminus \{x\}$. If $\Aut(\Cay(G,S))_1$ fixes $S\cap Bx$ setwise,  then $\Aut(\Cay(G,S))_1$ fixes $x$.
\end{corollary}

\begin{proof}
 Assume that, for every  $\varphi \in \Aut(\Cay(G,S))_1$, $(S\cap Bx)^\varphi=S\cap Bx$.
Observe first that $B^x$ is uniquely determined setwise as the set of vertices that have an outneighbour in $S\cap Bx$ via some element of $S \cap Bx$. Furthermore, $X:=Tx\cup \{x\}$ is uniquely determined as the set of vertices that are not in $S$, but are outneighbours of some vertex in $B^x$ via some element of $S \cap Bx$. Now Lemma~\ref{mut-innbrs} (applied with the set $S$ replaced by $G\setminus S$) tells us that $x$ is the only vertex in $X$ having the property that, for every other vertex $y \in X$, there are at least three vertices of $B^x$ that are not inneighbours of either $y$ or $x$. Thus $x^\varphi=x$, for every $\varphi\in \Aut(\Cay(G,S))_1$.
\end{proof}

Finally, we can describe a situation that will arise frequently in which $S \cap Bx= (Bx \setminus Tx)\setminus\{x\}$ is fixed setwise by every automorphism $\varphi$ of $\Cay(G,S)$ that fixes $1$. We show that in this situation, $B$ and $x$ are fixed pointwise by $\varphi$, which is a significant step toward proving that $\Cay(G,S)$ is an $\Orr$.

\begin{prop}\label{B-distinct}
Let $G$ be a group, and suppose that $B<G$ where $B$ is an elementary abelian group of rank $k \ge 6$, and $x \in G$  centralises $B$. Let $T$ be the Imrich generating set for $B$ with respect to some generating tuple, and let $S$ be a generating set for $G$ such that $S=[(Bx\setminus Tx)\setminus\{x\}]\cup X$, where $X \subseteq G\setminus \langle B,x\rangle$ with $|X| \le 17$. Then $\Aut(\Cay(G,S))_1$ fixes $B$ and $x$ pointwise. 
\end{prop}

\begin{proof}
We claim that a vertex of $\Cay(G,S)$ is an outneighbour of at least $2^k-4k-4$ vertices of $S$ if and only if it is an element of $Bx^2$. 

Let $bx^2$ be an arbitrary element of $Bx^2$. Then, for any $c \in (B\setminus T)\setminus \{1\}$, we have $bc \in B$, so $(bcx)(cx)=bx^2$ is an outneighbour of $cx$ unless $bc \in T \cup \{1\}$; that is, unless $c \in b(T \cup \{1\})$. Since there are $2k+2$ elements in $b(T\cup 1)$, there are at most $2k+2$ vertices in $(Bx\setminus Tx)\	\setminus \{x\}$ that are not inneighbours of $bx^2$. This means that there are at least $2^k-4k-4$ vertices in $S$ that are inneighbours of $bx^2$.

On the other hand, if we take a vertex $v$ that is not in $Bx^2$ and it is an outneighbour of at least $2^k-4k-4$ vertices of $S$, then since $2^k-4k-4>17$, $v$ must be an outneighbour of some vertex $bx \in Bx \cap S$. As we observed above, outneighbours of vertices of $Bx$ via elements of $S \cap Bx$ are in $Bx^2$, so $v\notin Bx^2$ implies that $v$ is an outneighbour of at most $17$ vertices that are in $Bx$. Therefore, $v$ is an outneighbour of at most $34$ vertices of $S$ (at most $17$ in $Bx$ and at most $17$ in $X$). Since $2^k-4k-4>34$, this completes the proof of our claim.

Let $\varphi\in\Aut(\Cay(G,S))_1$. The previous paragraph shows that $(Bx^2)^\varphi=Bx^2$. Since $Bx$ is  the set of vertices of $\Cay(G,S)$ such that all but at most 17 of their outneighbours lie in $Bx^2$, we also have $(Bx)^\varphi=Bx$. Thus $(S\cap Bx)^\varphi=S\cap Bx=(Bx\setminus Tx)\setminus \{x\}$. 

Applying Corollary~\ref{cor-innbrs} to our graph, we see that $x^\varphi=x$, and so $(Tx)^\varphi=Tx$. Since $T$ is the connection set for a GRR on $B$ and $x$ is fixed, this implies that $Bx$ is fixed pointwise, which means that (using Lemma~\ref{Watkins-Nowitz}) $B$ is fixed pointwise.
\end{proof}

We will be using this proposition for most groups $G$ arising from Theorem~\ref{main} that contain an elementary abelian  subgroup of low index. As long as $G$ is sufficiently large, the elementary abelian subgroup will have high enough rank to allow us to apply this result. For small orders of $G$, we will make use of computations in \texttt{magma} to complete the proof.

\section{The groups in Theorem~$\ref{main}$~{\bf(ii)}}\label{s:ii}
In this section we deal with the groups arising from Theorem~\ref{main}~{\bf (ii)}. First, we give a result that finds an $\Orr$ for all but three specific families of groups. 

\begin{prop}\label{prop:reduction}
Let $G$ be a $2$-group with an abelian subgroup $A$, with a normal subgroup $N$ and with two elements $g\in G\setminus N$ and $n\in N\setminus A$ such that
$$A<N<G,\, |G:N|=|N:A|=2,\,g^2=1,\,n^g=n^{-1},\, \textrm{ and }a^g=a^{-1}, \textrm{ for each }a\in A.$$
Then one of the following holds:
\begin{itemize}
\item[(a)] $G$ has an $\Orr$;
\item[(b)] $\ell$ and $\kappa$ are non-negative integers, $V$ is an elementary abelian $2$-group of rank $2\ell +\kappa$ with generating set $$\{v_1,w_1,\ldots, v_\ell,w_\ell,e_1,\ldots, e_\kappa\},$$ and $G=\langle V , x\rangle$, where  $v_i^x=w_i$, $w_i^x=v_i$ for $i \in\{1,\ldots, \ell\}$ and $e_i^x=e_i$ for $i\in\{1,\ldots, \kappa\}$; and
\begin{itemize}
\item[(i)] $x^2=1$; or
\item[(ii)] $\kappa\ge 1$ and $x^2=e_1$;
\end{itemize}
\item[(c)] $G$ contains a maximal subgroup $D$ isomorphic to $D_4 \times C_2^\ell$ for some $\ell \in \mathbb N$ ($D_4$ is the dihedral group of order $8$); or
\item[(d)] $G$ is generalised dihedral.
\end{itemize}
\end{prop}

\begin{proof}
  We assume that conclusion (d) does not hold, that is, $G$ is not generalised dihedral.

Suppose that $A$ has exponent 2. Then $V:=\langle A,g\rangle$ is a maximal subgroup of $G$ and is elementary abelian. Now, $V$ is a $\mathbb{Z}\langle n\rangle$-module. Therefore,  from the structure theorem of finitely generated $\mathbb{Z}\langle n\rangle$-modules, we may write $V=A_1\times A_2\times \cdots \times A_{\ell+\kappa}$, where $A_i:=\langle v_i\rangle$ is a non-identity non-trivial cyclic $\mathbb{Z}\langle n\rangle$-module for each $i\in\{1,\ldots,\ell\}$, and $A_{\ell+i}:=\langle e_i\rangle$ is a non-identity trivial (and hence cyclic) $\mathbb{Z}\langle n\rangle$-module for each $i\in \{1,\ldots,\kappa\}$.  For each $i\in \{1,\ldots,\ell\}$, write $w_i:=v_i^n$. As $n^2\in A$ centralises $V$,
$v_1,w_1,v_2,w_2,\ldots,v_\ell,w_\ell,e_1,\ldots,e_\kappa$ is a basis of $V$ (as a vector space over the field of cardinality $2$) and $|V|=2^{2\ell+\kappa}$. 

Now the action of $n$ on $V$ is encoded in the pair $(\ell,\kappa)$. Thus, to determine the isomorphism class of $G$ it suffices to determine $n^2$. Write $n^2=v_1^{\varepsilon_1}w_1^{\varepsilon_1'}\cdots v_\ell^{\varepsilon_\ell}w_\ell^{\varepsilon_\ell'}e_1^{\eta_1}\cdots e_\kappa^{\eta_\kappa}$, for some $\varepsilon_1,\varepsilon_1',\ldots,\varepsilon_\ell,\varepsilon_\ell',\eta_1,\ldots,\eta_\kappa\in \{0,1\}$. Since $n$ centralizes $n^2$ and $v_{i}^n=w_{i}$, $w_i^n=v_i$, we deduce that $\varepsilon_i=\varepsilon_i'$ for every $i\in \{1,\ldots,\ell\}$. Let $I:=\{i\in \{1,\ldots,\ell\}\mid \varepsilon_i=1\}$ and write $n':=n\prod_{i\in I}v_i$. Now, $n$ and $n'$ induce the same action by conjugation on $V$ and $(n')^2=e_1^{\eta_1}\cdots e_\kappa^{\eta_\kappa}$. In particular, replacing $n$ by $n'$ if necessary, we may assume that $\varepsilon_i=0$ for each $i\in \{1,\ldots,\ell\}$.

If $\eta_1=\cdots=\eta_\kappa=0$, then $n^2=1$ and $G$ splits over $V$. Taking $x:=n$, this is conclusion (b)(i).

If $\eta_i\ne 0$ for some $i\in \{1,\ldots,\kappa\}$, then $n^2\ne 1$ and $e_1,\ldots,e_{i-1},n^2,e_{i+1},\ldots,e_\kappa$ are linearly independent and span $\langle e_1,\ldots,e_\kappa\rangle$. Therefore, up to a change of basis in $\langle e_1,\ldots,e_\kappa\rangle$, we may assume that $n^2=e_1$. This is conclusion (b)(ii).

Suppose that $A\cong C_4\times C_2^\ell$, for some $\ell\in\mathbb{N}$. Now $D:=\langle A,g\rangle$ is generalised dihedral isomorphic to $D_4\times C_2^\ell$. This is conclusion (c).

For the rest of the proof, we may assume that $A$ is neither elementary abelian nor isomorphic to $A\cong C_4\times C_2^\ell$, for some $\ell\in\mathbb{N}$. By Lemma~\ref{abelian-orr}, there exists $T\subseteq A$ with $|T|\ge 2$ such that $\Cay(A,T)$ is an $\Orr$ and such that the induced subgraph of $\Cay(A,T)$ on $T$ is weakly connected. As $G/A$ is elementary abelian of order $4$ and $G$ is not generalised dihedral, by Lemma~\ref{l:1} applied to the normal subgroup $A$, there exist $a_1,a_2\in G\setminus N$ with $G=\langle a_1,a_2,N\rangle$, $o(a_1)>2$ and $o(a_2)>2$.  Since $G\setminus A=nA \cup gA \cup ngA$ and every element of $gA$ is an involution, the set $\{a_1,a_2\}$ consists of one element of $nA$ and one of $ngA$. Since $G/A$ is elementary abelian and $A$ is abelian, $a_1^2$ and $a_2^2$ are in $A$ and both centralise $A$. Now $a_1a_2 \in gA$ and no element of $gA$ centralises $A$, because $A$ does not have exponent $2$. Thus, the second hypothesis of Lemma~\ref{l:1} holds, and we deduce that $G$ admits an $\Orr$; that is, conclusion~(a) holds.
\end{proof}

We break the rest of this section into subsections, each dealing with one of the three families of groups described in this proposition that are not generalised dihedral, but have not yet been shown to admit $\Orr$s.

\subsection{The groups of Proposition~\ref{prop:reduction}(b)(i).}

Let $G$ be a group that arises in Proposition~\ref{prop:reduction}(b)(i). Then we have $G=V\rtimes\langle x\rangle$, where $V=\langle v_1, w_1, \ldots, v_\ell, w_\ell, e_1, \ldots, e_\kappa\rangle$ is an elementary abelian $2$-group, $x^2=1$, $e_i^x=e_i$ for every $i\in\{1,\ldots, \kappa\}$, and $v_i^x=w_i$ for every $i\in\{1,\ldots,\ell\}$. 

We begin by observing that, if $\ell \le 1$, then $G$ is a generalised dihedral group.

\begin{lemma}\label{ell=1}
Let $G$ be a group that arises in Proposition~$\ref{prop:reduction}(b)(i)$. If $\ell \le 1$, then $G$ is generalised dihedral.
\end{lemma}

\begin{proof}
If $\ell=0$, then $G$ is elementary abelian, which is a special form of generalised dihedral.
If $\ell=1$, then $V'=\langle v_1x,e_1, \ldots, e_\kappa\rangle$ is an abelian group having index $2$ in $G$ (since $(v_1x)^2=v_1w_1$). Also, the action of $w_1$ by conjugation inverts every element of $V'$. Thus, $G$ is a generalised dihedral group over $V'$.
\end{proof}

We now show that there is an $\Orr$ for any remaining sufficiently large group of this type.

\begin{lemma}\label{biproof}
Let $G$ be a group that arises in Proposition~$\ref{prop:reduction}(b)(i)$. If $\ell \ge 2$ and $2\ell+\kappa \ge 8$,  then $G$ has an $\Orr$.
\end{lemma}

\begin{proof}
We will use two different generating sets, according to whether or not $\kappa \ge 2$, or $\kappa \in \{0,1\}$.

If $\kappa \ge 2$, then let $$V_1:=\langle v_2, \ldots, v_\ell, w_2, \ldots, w_\ell, e_1, \ldots, e_{\kappa}\rangle$$ and let $T_1$ be the Imrich generating set for $V_1$ with respect to this generating tuple (this exists because the rank of $V_1$ is $2(\ell-1)+\kappa\ge 6$). 
If $\kappa\in\{0,1\}$, then $2\ell+\kappa \ge 8$ implies $\ell \ge 4$. In this case, let $$V_2:=\langle v_2,v_4w_4,v_3,v_2w_2,v_4,v_2w_3,v_5,\ldots, v_\ell,e_1,\ldots, e_\kappa\rangle$$  and let $T_2$ be the Imrich generating set for $V_2$ with respect to this generating tuple (this exists because the rank of $V_2$ is at least $6$ since the first $6$ elements are always in the tuple). 

Let
$$S:=[(V_i\setminus T_i)\setminus \{1\}]v_1x\cup \{v_2x\},$$
where $i=1$ if $\kappa \ge 2$ and $i=2$ if $\kappa \in\{0,1\}$.

Observe that the product of any two elements of $S$ has the form $vv_1xv'v_1x=vv_1(v')^xw_1\in V_1v_1w_1$ (where $v, v' \in V_i$), or $vv_1xv_2x=vv_1w_2\in V_1v_1$ (where $v \in V_i$), or $v_2xvv_1x=v_2v^xw_1\in V_1w_1$ (where $v \in V_i$), or $(v_2x)^2=v_2w_2$. Since none of these can be $1$, $S$ is the connection set for an oriented digraph.

\smallskip
\noindent\textsc{Claim:} $v_2x$ is the only vertex of $S$ for which all of its outneighbours have no other inneighbour from $S$.

\noindent

\noindent\textit{Proof of the Claim: }Observe that the outneighbours of $v_2x$ are all in $V_1v_1 \cup V_1$, while the outneighbours of any vertex of $V_iv_1x$ are in $V_1v_1w_1\cup V_1w_1$. Thus every outneighbour of $v_2x$ has a unique inneighbour from $S$. It remains to prove the uniqueness.

Let $vv_1x \in S$ be arbitrary, so that $v \in V_i$ but $v \notin T_i\cup\{1\}$. If $i=1$, then let $f_1:=e_1$ and $f_2:=e_2$; if $i=2$, then let $f_1:=v_2w_2$ and $f_2:=v_3w_3$.
We claim that $v \notin (T_i\cup\{1\})f_1\cap (T_i\cup\{1\})f_2$. Otherwise, $v=tf_1=t'f_2$ for some $t, t' \in T_i\cup \{1\}$ implies $tt'=f_1f_2$. Due to the structure of $T_2$, this has no solutions if $i=2$. When $i=1$, the only solutions are $\{t,t'\}=\{1,e_1e_2\}$ or $\{t,t'\}=\{e_1,e_2\}$. Thus $v \in \langle e_1,e_2\rangle \subset T_1\cup \{1\}$, a contradiction that completes the proof of the claim. 

Let $j \in \{1,2\}$ be such that $v\notin (T_i\cup\{1\})f_j$. Then $vf_j \notin T_i\cup \{1\}$, so $vf_jv_1x \in S$ and $(vf_jv_1x)^2=v_1w_1vv^x$ is an outneighbour of $vf_jv_1x$.
It is also an outneighbour of $vv_1x$ since  $(vv_1x)^2=v_1w_1vv^x$.~$_\blacksquare$

\smallskip

Let $\varphi \in \Aut(\Cay(G,S))_1$.  From the previous claim, $\varphi$ fixes $v_2x$.
This implies that $V_iv_1x \cap S$ is fixed setwise by $\varphi$. Now we can apply Corollary~\ref{cor-innbrs} to see that $\varphi$ fixes $v_1x$. This also implies that $(T_iv_1x)^\varphi=T_iv_1x$, and our choice of $T_i$ as the generating set for a GRR on $V_i$ implies that $V_iv_1x$ is fixed pointwise by $\varphi$, and therefore so is $V_i$. Thus $\varphi$ fixes pointwise the generating set $\{v_2x,v_1x\}\cup V_i$ of $G$.  By Lemma~\ref{Watkins-Nowitz}, $\Cay(G,S)$ is an $\Orr$.
\end{proof}

This shows that every group that arises in Proposition~\ref{prop:reduction}(b)(i) with $\ell \ge 2$ that has order at least $2^9$ has an ORR.

\subsection{The groups of Proposition~\ref{prop:reduction}(b)(ii).}

Let $G$ be a group that arises in Proposition~\ref{prop:reduction}(b)(ii). Then we have $G=\langle V, x\rangle$, where $V=\langle v_1, w_1, \ldots, v_\ell, w_\ell, e_1, \ldots, e_\kappa\rangle$ is an elementary abelian $2$-group, $x^2=e_1$, $e_i^x=e_i$ for every $i\in\{ 1,\ldots, \kappa\}$, and $v_i^x=w_i$ for every $i \in\{1,\ldots, \ell\}$. 

\begin{lemma}\label{biiproof}
Let $G$ be a group that arises in Proposition~$\ref{prop:reduction}(b)(ii)$. If $2\ell+\kappa \ge 7$, then $G$ has an $\Orr$.
\end{lemma}

\begin{proof}
Let $(v_1, \ldots, v_\ell, w_1, \ldots, w_\ell, e_2, \ldots e_\kappa)$ be a generating tuple for  $$V':=\langle
v_1, \ldots, v_\ell, w_1, \ldots, w_\ell, e_2, \ldots e_\kappa
\rangle <V,$$ and let $T$ be the Imrich generating set for $V'$ with respect to this generating tuple (this exists because $V'$ has rank $2\ell+\kappa-1 \ge 6$). Let 
$$S:=Tx\cup\{x\}.$$

As the product of any two elements of $S$ is not $1$ (since it will be in $V'e_1$), $\Cay(G,S)$ is an oriented Cayley digraph.

By Lemma~\ref{mut-innbrs}, $x$ is the unique vertex of $S$ that has at least three mutual inneighbours with every other vertex of $S$. Since any automorphism $\varphi \in \Aut(\Cay(G,S))_1$ fixes $S$ setwise, we must have $x^\varphi=x$. Furthermore $(Tx)^\varphi=Tx$ and since $T$ is the generating set for a GRR on $V'$, any automorphism that fixes $x$ and fixes $Tx$ setwise must actually fix $\langle T\rangle x=V'x$ pointwise. Since $\varphi$ fixes every point of $\{x,V'x\}$ and $\langle x,V'x\rangle=G$, by Lemma~\ref{Watkins-Nowitz} we have $\varphi=1$ and this completes the proof.
\end{proof}

This shows that every group  that arises in Proposition~\ref{prop:reduction}(b)(ii) that has order at least $2^8$ has an ORR.

\subsection{The groups of Proposition~\ref{prop:reduction}(c).}

\begin{notation}\label{F3-notn}
Let $G$ be a group that arises in Proposition~\ref{prop:reduction}(c). Then $G$ has a maximal subgroup $D$ isomorphic to $D_4 \times C_2^\ell$ for some $\ell \in \mathbb N$. Studying the initial description of the groups in Proposition~\ref{prop:reduction}, we see that we may assume that
 $G=\langle A,g,n\rangle$, where $A\cong C_4 \times C_2^k$ is generated by $a_1,\ldots, a_{k+1}$ with $o(a_1)=4$ and $o(a_i)=2$ for $ i\in\{2,\ldots, k+1\}$, and $g$ acts by conjugation on $A$ inverting every $a_i$, $g^2=1$, $n^g=n^{-1}$, $n$ normalises $A$ and $n^2 \in A$. Furthermore, if $n$ centralises $A$, then $N:=\langle A,n\rangle$ is abelian and $G=\langle N,g\rangle$ is generalised dihedral. Thus, if we assume that $G$ is not generalised dihedral, $n$ does not centralise $A$.
\end{notation}


\begin{lemma}\label{not-centralised}
Let $G$ be as in Notation $\ref{F3-notn}$. If $G$ is not generalised dihedral, then
there exists $a \in A$ such that $o(a)>2$ and $a$ is not centralised by $n$. 
\end{lemma}

\begin{proof} 
Assume that $G$ is not generalised dihedral. Therefore $n$ must not centralise $A$. Also, there is some $a_1 \in A$ with $o(a_1)>2$.
If $a_1$ is not centralised by $n$ then taking $a:=a_1$ we are done, so we may assume that $n$ centralises every element of $A$ of order greater than $2$. Since $n$ does not centralise $A$, there is some $b \in A$ such that $b$ is not centralised by $n$. Thus $o(b)=2$.  Let $c:=b^n$, so $c \neq b$ has order $2$.  Now, $a_1b$ has order $o(a_1)>2$, and $(a_1b)^n=a_1b^n=a_1c \neq a_1b$, so taking $a:=a_1b$ satisfies our conclusion.
\end{proof}

\begin{lemma}\label{not-inverted} 
Let $G$ be as in Notation $\ref{F3-notn}$.
If $o(n)=2$ and $G$ is not generalised dihedral, then there is some element $a \in A$ such that $o(a)>2$ and $a$ is not inverted by $n$.
\end{lemma}

\begin{proof}
Assume  $o(n)=2$. Let $a_1$ be an element of $A$ with $o(a_1)>2$.
If $a_1$ is not inverted by $n$ then taking $a:=a_1$ we are done, so we may assume that $n$ inverts every element of $A$ of order greater than $2$. Consider $a_1b$ for any $b$ of order 2 in $A$. As $o(a_1b)=o(a_1)>2$, we have $a_1^{-1}b=(a_1b)^{-1}=(a_1b)^n=a_1^nb^n=a_1^{-1}b^n$, and $b^n=b$. Thus $n$ centralises every element of order $2$, and  this implies that every element in $A$ is inverted by $n$. Hence $ng$ centralises $A$ and $\langle A,ng\rangle$ is abelian. Since the action of $g$ by conjugation inverts $ng$ (because $(ng)^g=n^gg=n^{-1}g=ng$)  and each element of $A$, we deduce that $G$ is a generalised dihedral group over $\langle A,ng\rangle$, a contradiction.
\end{proof}

We now have the key points for dealing with these groups. Our final lemma in this section shows that all sufficiently large groups that arise in Proposition~\ref{prop:reduction}(c) are either generalised dihedral or have $\Orr$s.

\begin{lemma}\label{cproof}
Let $G$ be as in Notation~$\ref{F3-notn}$, with $k \ge 6$. Then either $G$ is generalised dihedral, or $G$ admits an $\Orr$.
\end{lemma}

\begin{proof}
Suppose that $G$ is not generalised dihedral; let $B:=\langle a_2, \ldots, a_{k+1}\rangle$, and let $T$ be the Imrich generating set for $B$ with respect to some generating tuple. By Lemma~\ref{not-centralised}, there exists $a\in A$ with $o(a)>2$ and $a^n\ne a$. Let 
\[
S:=\begin{cases}
[(Ba\setminus Ta)\setminus\{a\}] \cup \{n^{-1},an^{-1}g,an\}&\textrm{when }a^n\ne n^{-2}a^{-1} \textrm{ and }o(n)\ne 2;\\
[(Ba\setminus Ta)\setminus \{a\}] \cup \{an^{-1}g,an\}&\textrm{when }a^n\ne a^{-1}\textrm{ and }o(n)=2;\\
[(Ba\setminus Ta)\setminus \{a\}] \cup \{an^{-1}g,n\}&\textrm{when }(a')^n=n^{-2}(a')^{-1} \textrm{ for every }a'\in A \textrm{ with }o(a')=4.
\end{cases}
\]
Observe that the three cases above do not cover all possibilies: in the remaining case  every element $a'$ of order $4$ in $A$ is either centralised by $n$, or has $(a')^n=n^{-2}(a')^{-1}$, and they are not all in the second category.  In the next paragraph, we study this latter possibility.

So there is $a \in A$ as in Lemma~\ref{not-centralised}, and $a^n=n^{-2}a^{-1}$. Notice that $a\ne a^n=n^{-2}a^{-1} $ implies that $n^2 \neq a^2$ in this case.  

\smallskip

\noindent \textsc{Claim: }There exist three elements $b_1,b_2,b_3\in B$ such that $ab_1,ab_1b_2,ab_1b_3$ are distinct and centralised by $n$.

\smallskip

\noindent\textit{Proof of claim: }Since not all elements of order $4$ of $A$ are in the second category, there exists an element of order $4$ centralised by $n$. As $A=\langle a\rangle\times B$, we may assume that this element is of the form $ab_1$, for some $b_1\in B\setminus\{1\}$. If $|\cent B n|>2$, then we may choose two distinct elements $b_2,b_3\in \cent B n\setminus\{1\}$, and now $n$ centralises the three distinct elements $ab_1,ab_1b_2,ab_1b_3$. Assume $|\cent B n|\le 2$.

Now, $B_0:=\langle a^2,B\rangle$ is characteristic in $A$ because it consists of all elements of order $2$ of $A$, and hence $B_0$ is $n$-invariant. Clearly, $|\cent {B_0}n|\le 2|\cent B n|\le 4$. As $n^2\in A$ centralises $A$, the mapping $f:B_0\to \cent {B_0}n$ defined by $x\mapsto xx^n$ is a homomorphism with kernel $\cent {B_0}n$. Therefore, from the first homomorphism theorem, $|B_0/\cent {B_0}n|\le |\cent {B_0}n|$ and $|B_0|\le |\cent {B_0}n|^2=16$. However, $|B_0|=2^{k+1}\ge 2^{6+1}$, a contradiction.~$_\blacksquare$

\smallskip

Let $b_1,b_2,b_3\in B$ be as in the previous claim and let
$$S:=[(Ba\setminus Ta)\setminus \{a\}] \cup \{an^{-1}g,ab_1n,ab_1b_2n,ab_1b_3n\}.$$


We claim that, in all four cases, $\Gamma:=\Cay(G,S)$ is an ORR for $G$.
First we show that $S$ is asymmetric and all of its elements have order greater than $2$. Certainly in every case the elements of $S\cap A$ have order $4$ and do not contain any inverse-closed pair. Any products that are not in $A$ will clearly not be $1$. The remaining pairwise products are $n^{-1}an=a^n\neq 1$; $ann^{-1}=a \neq 1$; $n^{-2}$ or $n^2$ (notice that, if $o(n)=2$, then Lemma~\ref{not-inverted} implies that $n \notin S$); $(an^{-1}g)^2=a(a^n)^{-1} \neq 1$ by our choice of $a$; $(an)^2=aa^nn^2\neq 1$ by our choice of $a$ in the cases where $an \in S$; and for every $b_i,b_j \in \{1,b_1,b_2,b_3\}$, we have $ab_1b_inab_1b_jn=a^2b_ib_jn^2\neq 1$ since $n^2 \neq a^2b_ib_j$. Thus $\Gamma$ is an oriented digraph.

Let $\varphi \in \Aut(\Gamma)_1$. By Proposition~\ref{B-distinct}, $a^\varphi=a$ and $\varphi$ fixes every element of $B$ pointwise. Therefore $\varphi$ fixes every element of $\langle a,B\rangle=A$ pointwise.

Now if $S$ has the first of the four possible forms, then $an^{-1}g$ is the only vertex of $\{an^{-1}g,an,n^{-1}\}$ that has only one outneighbour (via these three elements) in $A$, so $(an^{-1}g)^\varphi=an^{-1}g$. The outneighbours of $n^{-1}$ via $\{n^{-1},an\}$ are $n^{-2}$ and $a$, while the outneighbours of $an$ are $a^n$ and $aa^nn^2$. We know that $a^n \neq a$, and $a^n=n^{-2}$ implies $n^{-2}=a$, but then $a^n=a$, a contradiction. Thus $an$ is unique in having $a^n$ as an outneighbour, so $(an)^\varphi=an$. This completes the proof since $\varphi$ fixes every element of the generating set $\{an^{-1}g,an\}\cup A$ of $G$.

If $S$ has the second of the four possible forms, then $n^2=1$ and the outneighbours of $an^{-1}g$ via $an^{-1}g$ and $an$ are $a(a^n)^{-1}$ and $anan^{-1}g=aa^ng$, while the outneighbours of $an$ are $aa^n$ and $an^{-1}gan=a(a^n)^{-1}g$. Thus $an$ is the only one of these vertices that has $a(a^n) \in A$ as an outneighbour, so $(an)^\varphi=an$ and $(an^{-1}g)^\varphi=an^{-1}g$. This completes the proof since $\varphi$ fixes every element of the generating set $\{an,an^{-1}g\}\cup A$ of $G$.

If $S$ has the third of the four possible forms, then $(a')^n=n^{-2}(a')^{-1}$ for every $a'$ of order $4$ in $A$. The pairwise products of $\{an^{-1}g,n\}$ are $a(a^n)^{-1}=a(n^{-2}a^{-1})^{-1}=n^2a^2$, $an^{-2}g$, $n^{2}$, and $a^ng$. Since $a^2\neq 1$, $an^{-1}g$ and $n$ do not have the same outneighbours, so $n^\varphi=n$ and $(an^{-1}g)^\varphi=an^{-1}g$. This completes the proof since $\varphi$ fixes every element of the generating set $\{n,an^{-1}g\}\cup A$ of $G$.

Finally, if $S$ has the fourth of the four possible forms, then $an^{-1}g$ is the only vertex of $S\setminus Ba$ that has only one outneighbour in $A$, so $(an^{-1}g)^\varphi=an^{-1}g$. Of the remaining three vertices of $S\setminus Ba$, $ab_1n$ is the only one that has both $a^2b_2n^2$ and $a^2b_3n^2$ as outneighbours, so $(ab_1n)^\varphi=ab_1n$. This completes the proof since $\varphi$ fixes every element of the generating set $\{an^{-1}g,ab_1n\}\cup A$ of $G$.
\end{proof}

This shows that every group that arises in Proposition~\ref{prop:reduction}(c) that has order at least $2^{10}$ has an ORR.

\section{The groups in Theorem~$\ref{main}$~(iii)}\label{s:iii}

We discuss in some detail the work in~\cite{HeMa} and establish some notation that we use for the rest of this section.  Let $G$, $N$, $g$ and $n_0$ be as in Theorem~\ref{main}~{\bf (iii)}. Thus $|G:N|=2$, $g^2=1$, $N=H\cup n_0H$ where $H:=\{n\in N\mid n^g=n^{-1}\}$ (note that $H$ is a set and not necessarily a group), $|H|=|N|/2$ and $N$ has no automorphisms inverting more than half of its elements (according to Theorem~\ref{main}, $N$ is a group described in the Hegarty and MacHale paper, and hence $N$ admits an automorphism inverting precisely half of its elements, namely conjugation via $g$, and no automorphisms inverting more than half of its elements). We observe that our assumption about the existence and properties of $n_0$ is not amongst the assumptions of Hegarty and MacHale, and we will be able to use this assumption to eliminate some of the groups in their classification.

The classification of Hegarty and MacHale is very satisfactory, but not very easy to use in our application. These groups fall into ten isoclinism classes and, for each class, the authors give a very explicit description of a stem group in the class. (We denote by $X'$ the derived subgroup and by $\Zent X$ the centre of the group $X$.) We recall that the groups $X$ and $Y$ are isoclinic if there exist two group automorphisms $\varphi:X'\to Y'$ and $\psi:X/\Zent X\to Y/\Zent Y$ with
$$[x_1\Zent X,x_2\Zent X]^\varphi=[(x_1\Zent X)^\psi,(x_2 \Zent X)^\psi],\qquad\textrm{for every }x_1,x_2\in X.$$
For instance, the dihedral group of order $8$ and the quaternion group of order $8$ are isoclinic.  Being isoclinic is an equivalence relation coarser than the equivalence relation determined by the notion of group isomorphism. This means that finite groups are subdivided into isoclinism classes. Two groups in the same isoclinism class may not have the same order, for instance, $X$ and $X\times Z$ are isoclinic for every finite group $X$ and for every abelian group $Z$. It is well-known and also easy to prove that, for every group $X$, in the isoclinism class of $X$ there exists a group $Y$ with $\Zent Y\le Y'$. A group satisfying $\Zent Y\le Y'$ is said to be \textit{stem} group. The definition of isoclinism yields that the stem groups are precisely the groups of smallest possible order within their isoclinism class. It is quite unfortunate that two stem groups may be isoclinic but not necessarily isomorphic: consider again the example of the dihedral and the quaternion group of order $8$.

We can now explain in some detail the work of Hegarty and MacHale. (Just for this paragraph, we say that $X$ is half-inverting, if $X$ has an automorphism inverting half of its elements and no automorphisms inverting more than half of its elements.) As a by-product of their main theorem, they prove that if $X$ and $Y$ are isoclinic and $X$ is half-inverting, then $Y$ is half-inverting. Moreover, Hegarty and MacHale prove that half-inverting groups fall into $10$ distinct isoclinism classes. To describe these $10$ isoclinism classes, Hegarty and MacHale exhibit $10$ (non-isoclinic) stem groups: two of order $32$, six of order $64$ and two of order $128$. We  emphasise once again that this does not mean that there are only two stem groups of order $128$  that are half-inverting. (A computation with the computer algebra system \texttt{magma}~\cite{magma} shows that there are $5$, $55$ and $251$ half-inverting stem groups of order $32$, $64$ and $128$, respectively, up to isomorphism.)

Now that we have explained some details required for understanding the classification of Hegarty and MacHale, we can proceed to deal with the groups that arise in our context. 

\begin{prop}\label{iiiproof}
Let $G$ be a group that arises in Theorem~$\ref{main}$~{\bf (iii)}. Then $G$ admits an $\Orr$ except possibly if $N$ has an abelian subgroup $A$ with $|N:A|=4$, and $A$ is either isomorphic to $C_4 \times C_2^\ell$ for some $\ell \le 5$, or $A$ is elementary abelian of rank at most $7$.
\end{prop}

\begin{proof}
  By Theorem~\ref{main}~{\bf (iii)}, $G$ is a $2$-group with a normal subgroup $N$, $|G:N|=2$, and elements $g \in G\setminus N$ and $n_0 \in N$ with $g^2=1$; the action of $g$ by conjugation on $N$ inverts precisely half of the elements of $N$; and $N=H \cup n_0H$, where $H:=\{n \in N \mid n^g=n^{-1}\}$, and $|H|=|N|/2$.

From~\cite[Lemma~$1$]{HeMa} and its proof (see also the last paragraph of~\cite[page~$132$]{HeMa}), $N$ has an abelian subgroup $A$ with $|N:A|=4$ and with $a^g=a^{-1}$ for each $a\in A$. Some more information on the interaction between $A$ and $g$ and $N$ is available by reading the proof of Lemma~$1$ in~\cite{HeMa} (again, see also the last paragraph of page $132$ in~\cite{HeMa}). Let $1,x_2,x_3,x_4$  be left coset representatives for $A$ in $N$; thus $N=A\cup x_2A\cup x_3A\cup x_4A$. According to Hegarty and MacHale, there are only two cases:

\smallskip 

\noindent\textsc{Case I: }$g$ inverts some element in each coset of $A$ in $N$;

\smallskip

\noindent\textsc{Case II: }$g$ inverts some element in the cosets $A,x_3A$ and $x_4A$, and $g$ inverts no elements in the coset $x_2A$.

\smallskip

For the benefit of the reader, we have used the same subdivision into cases here, as the subdivision used in~\cite{HeMa}. Replacing $x_2,x_3,x_4$ by suitable coset representatives we may assume that $g$ inverts $x_3$ and $x_4$ and that $g$ also inverts $x_2$ in \textsc{Case I}. Let $x\in \{x_2,x_3,x_4\}$ in \textsc{Case I} and $x\in \{x_3,x_4\}$ in \textsc{Case II}. Let $y\in H\cap xA$. Then, $y=xa$ for some $a\in A$ and  $a^{-1}x^{-1}=(xa)^{-1}=y^{-1}=y^g=(xa)^g=x^ga^g=x^{-1}a^{-1}$. This shows that $xa=ax$. Therefore
\[
H=
\begin{cases}
A\cup x_2\cent A  {x_2}\cup x_3 \cent A {x_3}\cup x_4\cent A {x_4}&\textrm{in }\textsc{Case I},\\
A\cup x_3\cent A{x_3}\cup x_4\cent A{x_4}&\textrm{in }\textsc{Case II}.
\end{cases}
\]
Theorem \ref{main}~{\bf (iii)} gives us information about the element $n_0$ in addition to the structure that was studied by Hegarty and MacHale, and we have not used this in our analysis yet. We do so now.  Recall that $N=H\cup n_0H$,
 and $H\cap n_0H=\emptyset$ because $H$ has cardinality $|N|/2$.

Assume that \textsc{Case I} holds. Now $n_0A\subseteq n_0H$ and hence $n_0H$ contains a whole left coset of $A$ in $N$. Since $H$ contains elements from each left coset of $A$ in $N$, we get $n_0H\cap H\neq\emptyset$, a contradiction. Therefore $2$-groups in \textsc{Case I} do not arise in Theorem~\ref{main}~{\bf(iii)}. 

For the rest of the proof, we assume that \textsc{Case II} holds. From~\cite[line 11 from the bottom of page~$134$]{HeMa}, we have $A\lhd N$ and $N/A$ is elementary abelian of order $4$; therefore, $A\lhd G$ because $g$ also normalises $A$.

As $$2|A|=|N|/2=|H|=|A|+|\cent A {x_3}|+|\cent A{x_4}|,$$ we deduce  $|A:\cent A{x_3}|=|A:\cent A{x_4}|=2$. Clearly, $n_0\in x_2A$, otherwise, arguing as in the previous paragraph, we get $n_0H\cap H\ne\emptyset$. Replacing $x_2$ if necessary, we may assume that $n_0=x_2$. Thus $$n_0H=x_2A\cup x_2x_3\cent {A}{x_3}\cup x_2x_4\cent A{x_4}.$$ It is important to observe that since $A$ is normal in $N$ and $N/A$ is elementary abelian, we have $x_2x_3A=(x_2A)(x_3A)=x_4A$ and $x_2x_4A=(x_2A)(x_4A)=x_3A$; therefore $x_2x_3\cent A{x_3}$ is contained in the coset $x_4A$ and  $x_2x_4\cent A{x_4}$ is contained in the coset $x_3A$. Now, the condition $N=H\cup n_0H$ yields $x_3A=x_3\cent A{x_3}\cup x_2x_4\cent A{x_4}$. This implies that $A$ is the union of a coset of $\cent A{x_4}$ with a coset of $\cent A{x_3}$, and since each of these subgroups of $A$ has cardinality $|A|/2$, this can happen only when $\cent A {x_3}=\cent A{x_4}$. Since $N=\langle A,x_3,x_4\rangle$, we have $$\Zent N=\cent A{x_3}\cap\cent A{x_4}=\cent A{x_4}<A;$$ hence, $|N:\Zent N|=8$. Finally we observe that $G/A$ is elementary abelian of order $8$ because $g$ acts by conjugation inverting $x_3$ and $x_4$ and hence fixes the cosets $x_3A$ and $x_4A$.

\smallskip

Suppose now that $A$ is neither elementary abelian nor isomorphic to $C_4\times C_2^\ell$, for some $\ell\in\mathbb{N}$. From Lemma~\ref{abelian-orr}, there exists a subset $T$ of $A$ of cardinality at least $2$ with $\Cay(A,T)$ an $\Orr$ and such that the subgraph induced by $\Cay(A,T)$ on $T$ is weakly connected.

We claim that $A=\cent GA$, that is, no element of $G\setminus A$ centralises $A$. Observe that $$G \setminus A=x_3A \cup x_4A \cup x_3x_4A \cup gx_3A \cup gx_4A \cup gx_3x_4A.$$ From~\cite[Lemma~$1$]{HeMa}, $N$ has no abelian subgroup of index less then four. Therefore no element in the cosets $x_3A$, $x_4A$ and $x_3x_4A$ centralises $A$ and,  moreover, $$\Zent N=\cent A{x_3}=\cent A {x_4}=\cent A{x_3x_4}.$$  Observe now that $\Zent N$ is not elementary abelian because $A$ has no elementary abelian subgroup of index $2$.  Since $g$ acts  by conjugation inverting each element of $\Zent N$, we deduce that no element in the cosets $gA$, $gx_3A$, $gx_4A$ and $gx_3x_4A$ centralises $\Zent N\le A$. Thus our claim is proved.

 By Lemma~\ref{l:1}, since $|G:A|=8$, there are generators $a_1,a_2,a_3$ for $G$ modulo $A$ none of which is an involution. Furthermore, since $G/A$ is elementary abelian and $A$ is abelian, $a_i^2$ centralises $A$ for every $i\in \{1,2,3\}$. From the previous paragraph, $a_ia_j$ does not centralise $A$, for every $i,j\in \{1,2,3\}$ with $i\ne j$. In particular, the second hypothesis of Lemma~\ref{l:1} holds, and we deduce that $G$ admits an $\Orr$.

\smallskip

 Suppose now that $A$ is either elementary abelian or isomorphic to $C_4\times C_2^\ell$ for some $\ell\in\mathbb{N}$.
We deal with the two possible structures for $A$ individually. 
 
Suppose first that $A$ is isomorphic to $C_4 \times C_2^\ell$. By ignoring the exceptions from our statement, we may assume that $\ell \ge 6$. Let $a_0\in A\setminus \Zent N$ with $o(a_0)=4$ (observe that this is possible because $\Zent N<N$ cannot contain all the elements of order $4$ of $A$). Let $B$ be an elementary abelian subgroup of rank $\ell$ in $A$ that does not include the non-identity square element of $A$, and let $T$ be the Imrich generating set for $B$ with respect to some generating tuple. 
Let $$X:=\{gx_3a_0, gx_3a_0b, gx_3a_0b',gx_3a_0b''\},$$ where $b$, $b'$, and $b''$ are chosen from $B$ so that $\langle b,b',b''\rangle$ has order $8$ and none of the pairwise products from $\{1,b,b',b''\}$ is $(a_0^{-1})^{x_3}a_0$: the rank of $B$ is easily large enough that such choices are possible. Let $$Y:=\{gx_4a_0,gx_4a_0c,gx_4a_0c'\},$$ where $c$ and $c'$ are chosen from $B$ so that $\langle c, c'\rangle$ has order $4$ and none of the pairwise products from $\{1,c,c'\}$ is $(a_0^{-1})^{x_4}a_0$ (again, this is possible because $B$ has rank at least $6$). Let $$S:=[(Ba_0\setminus Ta_0)\setminus\{a_0\}] \cup X\cup Y \cup \{gx_3x_4\}.$$

The elements of $S \cap Ba_0$ have order $4$ and none is the inverse of another since $a_0^2 \notin B$. The pairwise products of the elements of $S$ that are not in $A$ cannot yield the identity. Furthermore, observe that, for every $z, z' \in \{1,b,b',b''\}$, we have $$(gx_3a_0z)(gx_3a_0z')=(a_0^{-1})^{x_3}a_0zz'\neq 1$$ by our choice of $b$, $b'$ and $b''$ and the fact that $a_0 \notin \Zent N=\cent A{x_3}$. Similarly,  for every $z, z' \in \{1,c,c'\}$, we have $$(gx_4a_0z)(gx_4a_0z')=(a_0^{-1})^{x_4}a_0zz'\neq 1$$ by our choice of $c$ and $c'$ and the fact that $a_0 \notin \Zent N=\cent A{x_4}$.  Finally, $$(gx_3x_4)^2=gx_3x_4gx_3x_4=x_3^gx_4^gx_3x_4=x_3^{-1}x_4^{-1}x_3x_4 \neq 1,$$ where in the third equality we have used  $x_3, x_4 \in H=\{n\in N\mid n^g=n^{-1}\}$, and in the last inequality we used the fact that $x_3$ and $x_4$ do not commute because $x_3x_4\notin H$ and $$x_4^{-1}x_3^{-1}=(x_3x_4)^{-1}\ne (x_3x_4)^g=x_3^gx_4^g=x_3^{-1}x_4^{-1}.$$ 
This proves that $S\cap S^{-1}=\emptyset$ and $\Cay(G,S)$ is an oriented graph.

By Proposition~\ref{B-distinct}, any automorphism $\varphi \in \Aut(\Cay(G,S))_1$ fixes $B$ and $a_0$ pointwise, so fixes $\langle a_0, B\rangle =A$ pointwise. 

Observe that, for any of the elements of $X$, its outneighbours via elements of $X\cup Y\cup\{gx_3x_4\}$ have four outneighbours in $A$ and four not in $A$; the elements of $Y$ each have three outneighbours via $X \cup Y \cup \{gx_3x_4\}$ in $A$ and five not in $A$; and $gx_3x_4$ has only one outneighbour via $X \cup Y \cup \{gx_3x_4\}$ that is in $A$. Thus $X^\varphi=X$, $Y^\varphi=Y$, and $(gx_3x_4)^\varphi=gx_3x_4$, for every $\varphi\in\Aut(\Cay(G,S))_1$. Furthermore, in $X$, since $\langle b,b',b''\rangle$ has order $8$, $gx_3a_0$ is the only vertex that does not have either $(a_0^{-1})^{x_3}a_0bb'$ or $(a_0^{-1})^x_3a_0bb''$ as an outneighbour. Therefore $gx_3a_0$ is fixed by $\Aut(\Cay(G,S))_1$. Similarly, in $Y$, since $\langle c,c'\rangle$ has order $4$, $gx_4a_0$ is the only vertex that does not have $(a_0^{-1})^{x_4}a_0cc'$ as an outneighbour. Therefore $gx_4a_0$ is fixed by $\Aut(\Cay(G,S))_1$. Thus $\Aut(\Cay(G,S))_1$ fixes the generating set $\{gx_3x_4, gx_3a_0,gx_4a_0\}\cup A$ for $G$ pointwise, so by Lemma~\ref{Watkins-Nowitz}, $\Aut(\Cay(G,S))_1=1$ and $\Cay(G,S)$ is an $\Orr$.

\smallskip

Suppose now that $A$ is elementary abelian, that is, $A\cong C_2^{\ell}$ for some $\ell\in\mathbb{N}$.  By ignoring the exceptions from our statement, we may assume that $\ell\ge 8$. As $x_3$ and $x_4$ do not commute, the commutator $d:=x_4^{-1}x_3^{-1}x_4x_3\ne 1$. Recall that $N/\Zent N$ is elementary abelian and hence the commutator subgroup $N' $ is contained in $\Zent N$. Thus $d\in \Zent N$. Let $d,b_1, \ldots, b_{\ell-2}$ be an irredundant generating set for $\Zent N\cong C_2^{\ell-1}$, and let $B:=\langle b_1, \ldots, b_{\ell-2}\rangle$. So $B$ is elementary abelian of rank $\ell-2 \ge 6$.  Let $T$ be the Imrich generating set for $B$ with respect to some generating tuple. 

Let $a_0\in A\setminus \Zent N$ and observe that $\langle\Zent N,g,x_3a_0,x_4a_0\rangle$ is a subgroup of $G$ having index $2$. Since $G$ is not generalised dihedral, there exists $v\in G\setminus \langle\Zent N,g,x_3a_0,x_4a_0\rangle$ with $o(v)>2$.
Let $$X:=\{gx_3a_0, gx_3a_0b, gx_3a_0b',gx_3a_0b''\},$$ where $b$, $b'$, and $b''$ are chosen from $B$ so that $\langle b,b',b''\rangle$ has order $8$ and none of the pairwise products from $\{1,b,b',b''\}$ is $(a_0^{-1})^{x_3}a_0$: the rank of $B$ is easily large enough that such choices are possible. Let $$Y:=\{gx_4a_0,gx_4a_0c,gx_4a_0c'\},$$ where $c$ and $c'$ are chosen from $B$ so that $\langle c, c'\rangle$ has order $4$ and none of the pairwise products from $\{1,c,c'\}$ is $(a_0^{-1})^{x_4}a_0$ (again, this is possible because $B$ has rank at least $6$). Let $$S:=gx_3x_4[(B \setminus T)\setminus\{1\}]\cup X \cup Y \cup \{v\} \cup \{gx_3x_4d\}.$$ 

For any $z, z' \in B$, we have $$(gx_3x_4z)(gx_3x_4z')=x_3^{-1}x_4^{-1}x_3x_4zz'
=dzz'\ne 1,$$ because $d \notin B$. The argument that no two elements of $X$ or $Y$ are inverses of one another is exactly as in the previous case. Notice also that $v$ is not in $\langle S \setminus\{v\}\rangle$, so $o(v)>2$ implies that $v^{-1} \notin S$. Now, observe that, for any $b \in B \setminus\{1\}$, we have $$gx_3x_4bgx_3x_4d=d^{x_4}bd=dbd=b\neq 1$$ since $d \in \Zent N$ which is elementary abelian. Finally, $(gx_3x_4d)^2=d \neq 1$. Thus $S\cap S^{-1}=\emptyset$ and $\Cay(G,S)$ is an oriented Cayley graph. 

Observe that $(gx_3x_4)^2=d$ implies $o(gx_3x_4)=4$, and since $\Zent N$ is elementary abelian we have $B<\Zent N=\Zent G$, so $gx_3x_4$ centralises $B$. Thus Proposition~\ref{B-distinct} applies to show that $\Aut(\Cay(G,S))_1$ fixes $B$ and $gx_3x_4$ pointwise.
Thus the cosets  of $B$ are blocks of imprimitivity for $\Aut(\Cay(G,S))_1$. Since the four vertices of $X$ lie in one coset of $B$, the three vertices of $Y$ lie in a different coset, and the vertices $v$ and $gx_3x_4d$ each lie in a different coset of $B$, we have $X^\varphi=X$ and $Y^\varphi=Y$, for every $\varphi\in \Aut(\Cay(G,S))_1$. The same argument as in the previous case again shows that $gx_3a_0$ and $gx_4a_0$ are fixed by $\Aut(\Cay(G,S))_1$.

Let $\varphi\in \Aut(\Cay(G,S))_1$. Since $d \in \Zent N$ has order $2$, we see that $gx_3x_4d$ has $2^\ell-2\ell-2$ outneighbours in $B$, but $v$ has at most one outneighbour in $B$. Thus $v^\varphi=v$ and $(gx_3x_4d)^\varphi=gx_3x_4d$. We now see that $\varphi$ fixes every point of the generating set $\{ gx_3a_0,gx_4a_0,v,gx_3x_4,gx_3x_4d\}\cup B$ for $G$. Since $\varphi$ was arbitrary, Lemma~\ref{Watkins-Nowitz} gives that $\Cay(G,S)$ is an $\Orr$.
\end{proof}

\section{Proof of Theorem~$\ref{conj}$}\label{s:conclusion}

We are now ready to prove our main theorem, Theorem~\ref{conj}.

\begin{proof}[Proof of Theorem~$\ref{conj}$]
If $|G|=2$, then $\Cay(G, \emptyset)$ is an $\Orr$ for $G$. We henceforth assume that $G$ is not generalised dihedral.

By Theorem~\ref{main}, $G$ admits an $\Orr$ unless $G$ is as in Theorem~\ref{main}~{\bf (ii)}, Theorem~\ref{main}~{\bf (iii)}, or $G \cong Q_8, C_3^2$, or $C_3\times C_2^3$. The final three possibilities are in our list of exceptions in the statement of Theorem~\ref{conj}.
Suppose now that $G$ is as in Theorem~\ref{main}~{\bf (ii)}. By Proposition~\ref{prop:reduction}, either $G$ admits an $\Orr$, or one of the following  possibilities holds:
\begin{enumerate}
\item[(1)] $\ell$ and $\kappa$ are non-negative integers, $V$ is an elementary abelian $2$-group of rank $2\ell +\kappa$ with generating set $\{v_1,w_1,\ldots, v_\ell,w_\ell,e_1,\ldots, e_\kappa\},$ and $G=\langle V,x\rangle$, where  $v_i^x=w_i$, $w_i^x=v_i$ for $i\in \{1,\ldots, \ell\}$ and $e_i^x=e_i$ for $i\in \{1,\ldots, \kappa\}$; and
\begin{itemize}
\item[(i)] $x^2=1$; or
\item[(ii)] $\kappa\ge 1$ and $x^2=e_1$; or
\end{itemize}
\item[(2)] $\langle A,g\rangle$ is isomorphic to $D_4 \times C_2^\ell$ for some $\ell \in \mathbb N$.
\end{enumerate}

If (1)(i) holds, then by Lemma~\ref{ell=1} we have $\ell\ge 2$. Now by Lemma~\ref{biproof}, $G$ admits an $\Orr$ as long as $2\ell+\kappa \ge 8$; equivalently, as long as $|G|\ge 2^9$.
If (1)(ii) holds, then by Lemma~\ref{biiproof}, $G$ admits an $\Orr$ as long as $2\ell+\kappa \ge 7$; equivalently, as long as $|G| \ge 2^8$.
If (2) holds, then by Lemma~\ref{cproof}, $G$ admits an $\Orr$ as long as $\ell \ge 6$; equivalently, as long as $|G| \ge 2^{10}$.
On the other hand, if $G$ is as in Theorem~\ref{main}~{\bf (iii)}, then by Proposition~\ref{iiiproof}, $G$ admits an $\Orr$ when $|G|\ge 2^{11}$.

It remains to check $2$-groups of order at most $2^9$ that are not generalised dihedral and that satisfy Theorem~\ref{main}~{\bf (ii)} and the $2$-groups of order at most $2^{10}$ that satisfy Proposition~\ref{iiiproof}.
 These were all checked with the aid of \texttt{magma}, and the only groups that do not admit an $\Orr$ are those listed.
\end{proof}

\end{document}